\documentclass[11pt]{article}

\usepackage{amsmath,amssymb,amsthm,mathrsfs}
\usepackage{graphicx,geometry,hyperref}
\usepackage{tcolorbox}
\usepackage{tikz}
\usepackage{subfigure}
\usepackage{float}
\usepackage{indentfirst}
\usepackage{appendix}
\newcommand{\pa}{\mathbb{P}_0}
\newcommand{\pn}{\mathbb{P}_{\neq}}
\newcommand{\Laps}{\Delta^{(\sigma)}}
\newcommand{\Lapsh}{\Delta^{(\sigma)}_H}
\newcommand{\nablas}{\nabla^{(\sigma)}}
\newcommand{\divs}{\operatorname{div}^{(\sigma)}}
\newcommand{\Lk}{\mathcal{L}_K}
\newcommand{\R}{\mathcal{R}}

\newtheoremstyle{break} 
    {3pt} 
    {3pt} 
    {\itshape} 
    {} 
    {\bfseries} 
    {.} 
    {5pt} 
    {\thmname{#1}\thmnumber{ #2}\thmnote{ (#3)}} 
\theoremstyle{break}

\newtheorem{theorem}{Theorem}[section]

\newtheorem{lemma}[theorem]{Lemma}

\newtheorem{proposition}[theorem]{Proposition}

\newtheorem{conjecture}[theorem]{Conjecture}
\newtheorem{remark}[theorem]{Remark}

\newcommand{\rref}[1]{\hyperref[#1]{(\ref*{#1})}}

\begin{document}
\title{{\bf Rigidity of Landau solution in the rotated self-similar class}}
\author{Yuxuan Shi and Liqun Zhang}
\date{}
\maketitle

\begin{abstract}
    Landau solution is a special family of solutions to the stationary Navier Stokes equations, which is important for the study of stationary problems such as asymptotic behavior and regularity. In this paper we prove that the Landau solution is rigid in rotated self-similar class when rotation parameter $\alpha$ sufficiently small or large. The arguments in the two parts rely on different approaches. For small rotation, we use the compactness argument and establish a classification lemma of the self-similar kernel of the linearization around any Landau solution, this method can also be extended to DSS and RDSS cases. For large rotation, the strong angular dissipation produces additional coercivity on the non-axisymmetric component, which forces the solution to reduce into Landau solution. 
\end{abstract}

\section{Introduction}

Consider the 3D stationary Navier Stokes (SNS) equations
\begin{equation}\label{eq:sns}
    \left\{\begin{aligned}
        &\,-\Delta u+u\cdot\nabla u+\nabla p=0\\
        &\qquad\quad\ \operatorname{div}u=0
    \end{aligned}\right.\quad\text{in } \mathbb{R}^3\backslash\{0\}
\end{equation}
satisfies 
\begin{equation}\label{eq:scale_bound}
    |u(x)|\le\frac{C}{|x|}
\end{equation}
for some $C>0$, which is a natural critical condition under the Navier-Stokes scaling
\begin{equation}\label{eq:scaling}
    u_{\lambda}(x)=\lambda u(\lambda x),\quad\ p_{\lambda}(x)=\lambda^2p(\lambda x)
\end{equation}
and by \cite{ST2000} (using scaling argument and regularity theory of SNS)
\begin{equation}\label{eq:derivative_decay}
    |\nabla^ku|\le\frac{C}{|x|^{k+1}}
\end{equation}

In fact, there is indeed a family of solutions exactly in this class, which is called Landau solution. This solution is independently computed by Slezkin \cite{S1934} in 1934, Landau \cite{L1944} in 1944 and Squire \cite{S1951} in 1951. More recently it is revisited and obtained under fewer restrictions by Tian-Xin \cite{TX1998} and Cannone-Karch \cite{CK2004}. The detailed computation can also be found in the textbook \cite[$\S 23$]{LL1987} of Landau and Lifshitz or \cite[$\S 8$]{T2018} of Tsai.

We use the standard spherical coordinates $(r,\theta,\phi)$ with
\[
    (x_1,x_2,x_3)=(r\sin\phi\cos\theta,r\sin\phi\sin\theta,r\cos\phi)
\]
and normal orthogonal basis
\[
    e_r=\frac{x}{r},\quad e_{\theta}=(-\sin\theta,\cos\theta,0),\quad e_{\phi}=e_{\theta}\times e_r
\]
to give the explicit formula of Landau solution.

In spherical coordinates, the Landau solution can be written as 
\begin{equation}\label{eq:Landau-solution}
    U^b=\frac{2}{r}\left(\frac{a^2-1}{(a-\cos\phi)^2}-1\right)e_r+\frac{-2\sin\phi}{r(a-\cos\phi)}e_{\phi},\qquad P^b=\frac{4(a\cos\phi-1)}{r^2(a-\cos\phi)^2}
\end{equation}
with parameter $a\in (1,\infty]$, they solve
\begin{equation}
    -\Delta U^b+U^b\cdot\nabla U^b+\nabla P^b=b\delta_0=\beta(a)e_3\delta_0
\end{equation}
where $e_3$ is unit vector in $x_3$-direction, $\delta_0$ is dirac function at the origin and
\[
    \beta(a)=16\pi\left(a+\frac{1}{2}a^2\log\frac{a-1}{a+1}+\frac{4a}{3(a^2-1)}\right)
\]

From \eqref{eq:Landau-solution}, we know that the Landau solution is not only an axisymmetric, no-swirl, but a self-similar (SS) solution (equivalently $-1$ homogeneous), which is invariant under scaling \eqref{eq:scaling}. Besides self-similarity, there are three other weaker types of similarity: RSS, DSS, RDSS.

We say $(u,p)$ is a rotated self-similar solution to SNS, if $(u,p)$ satisfies \eqref{eq:sns} and
\begin{equation}
    u(x)=\lambda \R(-2\alpha\log\lambda)u\left(\lambda\R(2\alpha\log\lambda)x\right),\quad\ \text{for}\ \forall\lambda>0\ \text{and some}\ \alpha\in\mathbb{R} 
\end{equation}
where $\R(\theta)=e^{\theta J}$ is rotation around $x_3$-axis with angle $\theta$, i.e.
\[
    \R(\theta)=\begin{pmatrix}
        \cos\theta & -\sin\theta & 0\\
        \sin\theta & \cos\theta & 0\\
        0 & 0 & 1
    \end{pmatrix}\qquad J=\begin{pmatrix}
        0 & -1 & 0\\
        1 & 0 & 0\\
        0 & 0 & 0
    \end{pmatrix}
\]
in brief, we call this solution a $\alpha$-RSS solution.

Moreover, we say $(u,p)$ is discretely self-similar solution ($\lambda_0$-DSS), if $\exists\lambda_0>1$, s.t.
\begin{equation}
    u(x)=\lambda_0u(\lambda_0x)
\end{equation}
and say $(u,p)$ is rotated discretely self-similar solution ($\lambda_0,\alpha_0$-RDSS) if $\exists\lambda_0>1,\ \alpha_0\in\mathbb{R}$ such that
\begin{equation}
    u(x)=\lambda_0 \R(-\alpha_0) u(\lambda_0\R(\alpha_0)x)
\end{equation}
the inclusion relation between four type of self similarity is
\[
    \text{SS}\subset\text{RSS}\subset\text{DSS}\subset\text{RDSS}
\]

\subsection{Rigidity problem for steady Navier Stokes equations and main result}

Landau solution is a singular solution to SNS \eqref{eq:sns} which is axisymmetric, self-similar and no-swirl. A natural problem is whether the Landau solution is rigid under some restrictions? In \cite{TX1998} Tian and Xin prove that Landau solution is rigid under axisymmetric $\&$ self-similar restriction, and in \cite{S2011} \v{S}ver\'ak remove the axisymmetric restriction and prove Landau solution is rigid in self-similar class. But if we further relax the restriction to only \eqref{eq:scale_bound}, the rigidity of Landau solution is still unknown and conjuctured by \v{S}ver\'ak in \cite{S2011}:

\begin{conjecture}\label{conj:Sverak}
    Each solution of \eqref{eq:sns} in $\mathbb{R}^3\backslash\{0\}$ satisfies
    \[
        |u(x)|\le\frac{C}{|x|}
    \]
    must be a Landau solution.
\end{conjecture}

For small $C$, the Conj \ref{conj:Sverak} is proved in \cite{MT2012} by Miura and Tsai. When $C$ is large, without smallness and similarity, this question is still open in 3D until now. Indeed, the conjecture of \v{S}ver\'{a}k can also be studied in two or higher dimensions: In \cite{S2011}, \v{S}ver\'{a}k prove that self-similar solution of SNS in $\mathbb{R}^n\backslash\{0\}$ can only be the Jeffrey-Hamel solution, Landau solution and trivial solution ($\equiv 0$) for $n=2,\ n=3,\ n\ge 4$, but when we relax the restriction to rotated self-similar solution, Guillod and Wittwer find another special class of 2D solutions which are RSS but not SS in \cite{GW2015}. In higher dimensional case ($n\ge 4$), Bang-Gui-Liu-Wang-Xie \cite{BGLWX2025} prove that under general condition $|u|\le C/|x|$, the only solution is $u\equiv 0$, i.e. the $n\ge 4$ version of Conj \ref{conj:Sverak} is true.

\begin{center}
\begin{tabular}{|c|c|c|c|}
\hline
 & self-similar & rotated self-similar & $|u|\le\frac{C}{|x|}$\\
\hline
2D & Jeffery-Hamel solution \cite{S2011}  & Guillod-Wittwer solution \cite{GW2015} & Open\\
\hline
3D & Landau solution \cite{S2011} & Open & Open \\
\hline
4D & Trivial solution ($\equiv 0$) \cite{S2011} & Trivial solution \cite{BGLWX2025} & Trivial solution \cite{BGLWX2025}\\
\hline
\end{tabular}
\end{center}

A sub-conjecture of Conj \ref{conj:Sverak} is that whether there exists a bifurcation from Landau solution, in this direction, Kwon and Tsai \cite{KT2021} provided both analytical and numerical evidences for no bifurcation in DSS + axisymmetric class.

Since there exists nontrivial RSS solution in 2D, and there is no nontrivial RSS solution in higher dimension (as a corollary of the result in \cite{BGLWX2025}), it's a natural question to ask is there nontrivial RSS solution in 3D, in this paper we try to partially answer this question when the rotated parameter $\alpha$ is sufficiently small or large.

\begin{theorem}\label{thm:RSS}
     Suppose that $u$ is an $\alpha$-RSS solution, then $\forall M>0$, $\exists\, \alpha_+(M)=C(M+1)$, $\alpha_-(M)>0$, if 
    \begin{equation}\label{eq:scale_bound}
        |u(x)|\le\frac{M}{|x|}
    \end{equation}
    then either $0\le\alpha<\alpha_{-}(M)$ or $\alpha>\alpha_{+}(M)$ implies $u$ is a Landau solution.
\end{theorem}

In DSS of RDSS case, we also prove the following result

\begin{theorem}\label{thm:DSS}
     Suppose that $u$ is a $\lambda$-DSS solution, then $\forall M>0$, $\exists\, \lambda_+(M)>1$, if 
    \begin{equation}\label{eq:scale_bound}
        |u(x)|\le\frac{M}{|x|}
    \end{equation}
    and $1<\lambda <\lambda_+(M)$, then $u$ is Landau solution
\end{theorem}

\begin{theorem}\label{thm:RDSS}
    Suppose that $u$ is a $\lambda,\alpha$-RDSS solution, then $\forall M>0$, $\exists\, \lambda_+(M,\alpha)>1$, if 
    \begin{equation}\label{eq:scale_bound}
        |u(x)|\le\frac{M}{|x|}
    \end{equation}
    and $1<\lambda <\lambda_+(M,\alpha)$, then $u$ is Landau solution.
\end{theorem}

\begin{remark}
    RSS solution is actually a special class of DSS solution via relation $\lambda=e^{\frac{\pi}{|\alpha|}}$, hence $\lambda\to 1^+$ correspond to $\alpha\to\infty$, i.e. the large rotation part is a corollary of Thm \ref{thm:DSS}, but the energy estimate argument in section 5 can give a quantitative description of $\alpha_+(M)$
\end{remark}

\subsection{Motivation of rigidity result from asymptotic behavior and regularity}

Landau rigidity problem is closedly related to the asymptotic behavior of stationary flows. \v{S}ver\'ak and Tsai \cite{ST2000} established optimal derivative decay \eqref{eq:derivative_decay} under the critical bound $u(x)=O(|x|^{-1})$. For $C$ small, Korolev and \v{S}ver\'ak \cite{KS2011} observed that the leading order term is some Landau solution $U^b$, with remainder $O(|x|^{-\alpha})$, $1<\alpha<2$ by perturbative method. Very recently, Jia and \v{S}ver\'ak \cite{JS2026} proved the endpoint case $\alpha=2$ for the next order, and obtained the refined $O(|x|^{-2})$ asymptotics around small Landau solutions by analyze the spectral property of linearized system on $\mathbb{S}^2$. Without smallness, even the identification of the leading order term remains open, and a general Landau rigidity result will provide this through a blow down argument, as in the higher dimensional setting \cite{BGLWX2025}.

Rigidity is also relevant to removable singularity/regularity problem of SNS, the theory of removable singularity is developed through the works of Dyer-Edmunds \cite{DE1970}, Shapiro \cite{S1974}, Choe-Kim \cite{CK2000} and others. A sharp form of removable singularity is obtained by Kim and Kozono \cite{KK2006}, they prove a smooth solution in $B_R\backslash\{0\}\subset\mathbb{R}^n, n\ge 3$ satisfies $u\in L^n$ of $u=o(1/|x|)$, then the singularity can be removed. The sharpness is from Landau solution as a counterexample. However, when the equation holds on the whole $B_R$, nonzero Landau solution is excluded by their point force $b\delta_0$. Hence in \cite{T2018}, Tsai mentioned a conjecture that if $|u|\le C/|x|$ or equivalently $u\in L^{3,\infty}$ without smallness, whether $u$ is regular? Similarly a general rigidity result will show the regularity through a blow up argument, see \cite{BGLWX2025} for higher-dimensional case.

\subsection{Strategy of main results}

When $u$ is self-similar solution, by decomposing $u$ into tangential and normal part
\[
    u=\frac{1}{r}(V(\sigma)+f(\sigma)n),\quad\ V\in\mathfrak{X}(\mathbb{S}^2),\ f\in C^{\infty}(\mathbb{S}^2)
\]
(where $\mathfrak{X}(\mathbb{S}^2)$ is set of smooth tangent vector fields on $\mathbb{S}^2$), \v{S}ver\'ak \cite{S2011} observe that $\omega$ (vorticity of $V$ on $\mathbb{S}^2$) satisfies a good structure
\begin{equation}\label{eq:vorticity_SS}
    -\Laps\omega+\divs(V\omega)=0
\end{equation}
which implies $\omega=0$ via the maximal principle type lemma in Appendix B, and the problem is transfer into the classification problem of conformal geometry equation on $\mathbb{S}^2$ using the equation of $f$.

But when it comes to the $\alpha$-RSS case, the extra rotation terms $-4\alpha^2\Lk^2V-2\alpha(1+f)\Lk V$ (where $\Lk$ is Lie derivative w.r.t $K=\partial_{\theta}$) in tangential equation destroy the structure of vorticity (the rotation terms in normal equation also make trouble), which makes the argument in \cite{S2011} failed. Fortunately, when the rotation parameter is small or large, there are two distinct approaches to force the solution to be Landau. 

Our proving strategy is divided into two parts:

(i) When parameter $\alpha$ is sufficiently small, we can view it as a bifurcation type problem near self-similar solution. Using compactness argument and taking tangential limit of the Landau manifold, we transfer the small $\alpha$ rigidity problem into the classification problem of self-similar kernel of the linearized operator
\[
    \mathcal{L}_{U^b}w=-\Delta w+U^b\cdot\nabla w+w\cdot\nabla U^b+\nabla q=0,\quad\ \operatorname{div} w=0\quad\ \text{in}\ \mathbb{R}^3\backslash\{0\}
\]
for any Landau solution $U^b$, and prove that
\[
    \operatorname{Ker}\mathcal{L}_{U^b}\,\cap\,\{w|w\ \text{is self-similar}\}=\operatorname{span}\{\partial_{b_i}U^b\}_{i=1,2,3}
\]
this approach can also be used in the DSS and RDSS cases.

(ii) When parameter $\alpha$ is sufficiently large, using the DSS result we can get a qualitative result, a more quantitative way is that the rotation term in the profile equation provide an extra angular dissipation
\begin{equation}\label{eq:angular_dissipation}
    4\alpha^2\|\Lk V\|_{L^2}^2,\quad\ 4\alpha^2\|Kf\|_{L^2}^2
\end{equation}
inspired by how Pineau-Vicol \cite{PV2026} ruling out the backward RSS blow up of nonstationary system with large rotation, we prove that the coercivity caused by \eqref{eq:angular_dissipation} will kill the non-axisymmetric part $V_{\neq}$ and $f_{\neq}$, which reduces the problem back to axisymmetric $\&$ self-similar case.

\begin{remark}
    Similar symmetry classes arise naturally and earlier in the non-existence of backward type blow up of non-stationary Navier Stokes. The backward self-similar type blow up is raised by Leray \cite{L1934}, and it was ruling out by Ne\v{c}as-R\r{u}\v{z}i\v{c}ka-\v{S}ver\'ak \cite{NRS1996} and Tsai \cite{T1998} under global and local energy condition. Moreover, Chae-Wolf \cite{CW2017} rule out backward-DSS blow up when $\lambda-1$ small, and recently Pineau-Vicol \cite{PV2026} rule out backwrad-RSS with small or large rotation. The main different of the stationary problem compare to non-stationary problem is that there exists indeed non-trivial (Landau) solution, so we need a classification theorem to Landau solution rather than a Liouville type theorem to trivial solution.
\end{remark}

\section{Formulation of profile equation of RSS-solution on $\mathbb{S}^2$}

In \cite{S2011}, by decompose $u$ into tangential part and normal part
\[
    u=\frac{1}{r}(V(\sigma)+f(\sigma)n),\quad\ V\in\mathfrak{X}(\mathbb{S}^2),\ f\in C^{\infty}(\mathbb{S}^2)
\]
and rewrite the SNS into tangential, normal part and incompressible condition as follows:
\begin{equation}\label{eq:SS-profile}
\left\{\begin{array}{c}
    -\Lapsh V+\nablas_VV+\nablas(P-2f)=0\\
    -\Laps f+\nablas_Vf-|V|^2-|f|^2-2P=0\\
    \divs V+f=0
\end{array}\right.\quad\text{on}\ \mathbb{S}^2
\end{equation}
$\Lapsh$, $\Laps$, $\nablas$, $\divs$ is differential opeartors on $\mathbb{S}^2$, see \cite{S2011} or Appendix A for its derivation.

Next we derive the profile equation of $\alpha$-RSS solution on $\mathbb{S}^2$, by introducing the similarity variable $t=\log r$, and decompose $u$ into tangential and normal part as usual, we have
\[
    u(r,\sigma)=\frac{1}{r}\left(V(t,\sigma)+f(t,\sigma)n\right),\qquad p(r,\sigma)=\frac{1}{r^2}P(t,\sigma)
\]
the addition term compare to \eqref{eq:SS-profile} is caused by $\partial_r$ apply on $t$-variable, e.g. the divergence is
\[
    0=\operatorname{div}u=\partial_r\left(\frac{f(t,\sigma)}{r}\right)+\frac{2f(t,\sigma)}{r^2}+\frac{1}{r^2}\divs V\ \Longrightarrow\ \divs V+f+f_t=0
\]
the convection term is
\begin{equation}\nonumber
\begin{aligned}
    (u\cdot\nabla)u&=\frac{1}{r}(V+fn)\cdot\nabla\left(\frac{1}{r}(V+fn)\right)\\
    &=\frac{1}{r^2}D_{V+fn}(V+fn)-\frac{1}{r^3}f(V+fn)\qquad D\text{ is connection of }\mathbb{R}^3\\
    &=\frac{1}{r^3}\left(\nablas_VV-|V|^2n-|f|^2n+\nablas_Vfn+fV_t+ff_tn\right)
\end{aligned}
\end{equation}
the viscous term is
\begin{equation}\nonumber
\begin{aligned}
    \Delta u & =\partial_r^2u+\frac{2}{r}\partial_ru+\frac{1}{r^2}\Laps_S u\\
    &=\frac{1}{r^3}\left(V_{tt}-V_t+\Lapsh V+2\nablas f\right)+\frac{1}{r^3}\left(f_{tt}+f_t+\Laps f\right)n
\end{aligned}
\end{equation}
and the pressure term is
\begin{equation}\nonumber
    \nabla p=\left(n\partial_r+\frac{1}{r}\nablas\right)\left(\frac{1}{r^2}P(t,\sigma)\right)=\frac{1}{r^3}\left(\nablas P-2Pn+P_tn\right)
\end{equation}

To conclude, we have
\begin{equation}
    \left\{\begin{aligned}
        &\quad\ \,-V_{tt}+(1+f)V_t-\Lapsh V+\nablas_VV+\nablas(P-2f)=0\\
        &-f_{tt}+(f-1)f_t-\Laps f+\nablas_Vf-|V|^2-|f|^2+P_t-2P=0\\
        &\qquad\qquad\qquad\qquad\ \; \divs V+f+f_t=0
    \end{aligned}\right.
\end{equation}

When $u$ be a rotated self-similar solution, we know that
\[
    u(x)=\lambda\mathcal{R}(-2\alpha\log\lambda)u\left(\lambda\mathcal{R}(2\alpha\log\lambda)x\right)\ \Longrightarrow\ u(x)=\frac{1}{r}\mathcal{R}(2\alpha\log r)u\left(\mathcal{R}(-2\alpha\log r)\sigma\right)
\]
denote
\[
    \R(2\alpha\log r)u(\R(-2\alpha\log r)\sigma)=V(t,\sigma)+f(t,\sigma)n
\]
where $V$ satisfies
\[
    V(t,\sigma)=\mathcal{R}(2\alpha t)V\left(\mathcal{R}(-2at)\sigma\right)
\]

By definition of Lie derivative
\[
    V_t=\frac{d}{dt}\mathcal{R}(2\alpha t)V\left(\mathcal{R}(-2\alpha t)\sigma\right)=-2\alpha\mathcal{L}_{K}V
\]
where $K=\partial_\theta$ is generated by rotation $\R$, i.e. 
\[
    K(\sigma)=J\sigma=e_3\times\sigma=\partial_\theta\sigma\ \Longrightarrow\ K=\partial_\theta
\]
then the profile equation of $\alpha$-RSS solution is
\begin{equation}\label{eq:RSS_profile}
    \left\{\begin{array}{c}
        -\Lapsh V-4\alpha^2\mathcal{L}_K^2V-2\alpha(1+f)\mathcal{L}_KV+\nablas_VV+\nablas(P-2f)=0\\
        -\Laps f-4\alpha^2K^2f-2\alpha(f-1)Kf+\nablas_Vf-|V|^2-|f|^2-2\alpha KP-2P=0\\
        \divs V+f-2\alpha Kf=0
    \end{array}\right.
\end{equation}

\section{Classification lemma and compactness argument when $\alpha$ small}

In this section we will first prove a classification lemma of the kernel of the linearized operator (in self-similar class), which is useful for ruling out not only $\alpha$-RSS when $\alpha$ small, but also $\lambda$-DSS when $\lambda$ small. When $\alpha$ is small, we can use compactness argument and take tangential limit to reduce the problem to the linear classification lemma.

\subsection{Classification lemma of the kernel of linearized operator}

We first studied the kernel of linearized operator around $U^b$
\[(\mathcal{L}_{U^b})w=-\Delta w+U^b\cdot\nabla w+w\cdot\nabla U^b+\nabla q,\quad\ \operatorname{div}w=0\]
since $U^b$ is a solution family varies with parameter $b$, hence if we take $U^{b+\varepsilon c}$ into the \eqref{eq:sns}, and check the derivative value at $\varepsilon=0$, we have:
\begin{equation}
\begin{aligned}
    \left.\frac{d}{d\varepsilon}\right|_{\varepsilon=0}\left(-\Delta U^{b+\varepsilon c}+U^{b+\varepsilon c}\cdot\nabla U^{b+\varepsilon c}+\nabla P^{b+\varepsilon c}\right)=&-\Delta w+U^b\cdot\nabla w+w\cdot\nabla U^b+\nabla Q\\
    =&\left.\frac{d}{d\varepsilon}\right|_{\varepsilon=0}(b+\varepsilon c)\delta_0=c\delta_0
\end{aligned}
\end{equation}
where
\[
    w=\sum\limits_{i=1}^{3}c_i\partial_{b_i}U^b
\]
so $\{\partial_{b_i}U^b\}_{i=1,2,3}\subset\operatorname{Ker}(\mathcal{L}_{U^b})$. The vector $\partial_{b_i}U^b$ corresponds to the Landau solution $U^b$ varies in $b$, and whether there exists a bifurcation from Landau solution is closedly related to whether there exists other element in the kernel of linearized operator. See also the study on $\operatorname{Ker}\mathcal{L}_{U^b}$ in \cite{JS2026} for $b$ small but without similarity assumption, and \cite{KT2021} for the DSS + axisymmetric case with evidence combining analytical proof and numerical experiment.

In this section, we prove that the $\operatorname{Ker}(\mathcal{L}_{U^b})\,\cap\, \{w|w\ \text{is self-similar}\}=\operatorname{span}\{\partial_{b_i}U^b\}_{i=1,2,3}$.
\begin{proposition}[Classification of the self-similar kernel of $\mathcal{L}_{U^b}$]\label{prop:classification} \
    
    If $(w,q)$ is a self-similar solution satisfies
    \begin{equation}\label{eq:linear_perturb}
    \left\{\begin{array}{c}
        -\Delta w+U^b\cdot\nabla w+w\cdot\nabla U^b+\nabla q=0\\
        \operatorname{div} w=0
    \end{array}\right.\qquad\text{in}\ \mathbb{R}^3\backslash\{0\}
    \end{equation}
    then every solution of \eqref{eq:linear_perturb} can be represent by linear combination $w=c_1\partial_{b_1}U^b+c_2\partial_{b_2}U^b+c_3\partial_{b_3}U^b$ and $q=c_1\partial_{b_1}P^b+c_2\partial_{b_2}P^b+c_3\partial_{b_3}P^b$, then
    \[
        -\Delta w+U^b\cdot\nabla w+w\cdot\nabla U^b+\nabla q=c\delta_0,\qquad c=(c_1,c_2,c_3)
    \]
    i.e. $\operatorname{Ker}(\mathcal{L}_{U^b})\,\cap\, \{w|w\ \text{is self-similar}\}=\operatorname{span}\{\partial_{b_i}U^b\}_{i=1,2,3}$ 
    
\end{proposition}
\begin{proof}
    Denote
    \[
        U^b(r,\sigma)=\frac{1}{r}\left(V(\sigma)+F(\sigma)n\right),\qquad w(r,\sigma)=\frac{1}{r}\left(v(\sigma)+f(\sigma)n\right)
    \]
    where $U^b$ is Landau solution satisfies
    \begin{equation}\label{eq:Landau_sol}
        d{V}^\flat=0,\qquad V=\nablas\varphi,\qquad F=-\Laps\varphi=2e^{\varphi}-2
    \end{equation}
    
    The (linear) perturbation equation on $\mathbb{S}^2$ is
    \begin{equation}\label{eq:perturb_equation}
    \left\{\begin{array}{c}
        -\Lapsh v+\nablas_{v}V+\nablas_Vv+\nablas(\widetilde{q}-2f)=0\\
        -\Laps f+\nablas_{V}f+\nablas_{v}F-2\langle v,V\rangle_{(\sigma)}-2fF-2\widetilde{q}=0\\
        \divs v+f=0
    \end{array}\right.
    \end{equation}

    Take curl (or $*d$) on the tangential equation of \eqref{eq:perturb_equation}, then by Prop \ref{prop:perturb_vorticity_equation} and \eqref{eq:Landau_sol} we get
    \begin{equation}\label{eq:vorticity_linear}
        -\Lapsh\omega+\divs(V\omega)=0
    \end{equation}
    where $*dv^\flat=\omega$, and by maximal type Lem \ref{lem:maximal_lemma} (from \cite{S2011}) we have $\omega\equiv 0$.

    Therefore, $\exists\,\psi$ such that $v=\nablas\psi$, then
    \[
        f=-\Laps\psi,\qquad F=-\Laps\varphi
    \]
    the tangential (first) equation of \eqref{eq:perturb_equation} becomes
    \begin{equation}\nonumber
        \begin{aligned}
            0&=-\Lapsh\nablas\psi+\nablas\left\langle\nablas\psi,\nablas\varphi\right\rangle_{(\sigma)}+\nablas(\widetilde{q}-2f)\\
            &=\nablas\left(\widetilde{q}-f+\left\langle\nablas\psi,\nablas\varphi\right\rangle_{(\sigma)}\right)
        \end{aligned}
    \end{equation}
    then
    \begin{equation}\label{eq:equation_with_const}
        \widetilde{q}-f+\left\langle v,V\right\rangle_{(\sigma)}=C
    \end{equation}

    Substitute \eqref{eq:equation_with_const} into the normal equation of \eqref{eq:perturb_equation}, we get
    \begin{equation}\label{eq:normal}
        -\Laps f+\divs\left(fV+Fv\right)-2f=2C
    \end{equation}

    Integrate \eqref{eq:normal} on $\mathbb{S}^2$ we know that $C=0$, and
    \begin{equation}
    \begin{aligned}
        0&=-\Laps f+\divs(fV)+\divs\left((F+2)v\right)\\
        &=-\Laps f+\divs(f\,\nablas\varphi)+\divs\left(2e^{\varphi}\nablas\psi\right)\qquad \text{by  \eqref{eq:Landau_sol}}\\
        &=-\Laps\left(f-2e^{\varphi}\psi\right)+\divs((f-2e^{\varphi}\psi)\,\nablas\varphi)
    \end{aligned}
    \end{equation}
    since the solution of $-\Laps\Psi+\divs(\Psi\nablas\varphi)=0$ must be of the form $Ce^\varphi$ (take $\Psi=C(x)e^{\varphi}$ into the equation, $C(x)$ must be constant from the strong maximal principle), we know
    \[
        f-2e^{\varphi}\psi=Ce^{\varphi}
    \]

    Substitute $\psi$ by $\psi-\frac{C}{2}$, we have
    \[
        -\Laps\psi=f=2e^{\varphi}\psi
    \]
    consider conformal transform $\widetilde{g}=e^{\varphi}g^{(\sigma)}$, then $\Delta^{\widetilde{g}}=e^{-\varphi}\Laps$, consequently
    \[
        -\Delta^{\widetilde{g}}\psi=2\psi
    \]
    
    Since $\widetilde{g}$ isometric to $g^{(\sigma)}$ (from the proof in Thm 1 of \cite{S2011}), and the Laplacian of isometry metric is equivalent under isometric transformation, hence the first non-zero eigenvalue of $-\Delta_{\widetilde{g}}$ is exactly $\lambda_1=2$ and $\operatorname{dim}\operatorname{Ker}(2+\Delta_{\widetilde{g}})=3$, i.e. $\operatorname{dim}\operatorname{Ker}(\mathcal{L}_{U^b})\,\cap\,\{w|w\ \text{is self-similar}\}=3$, since
    \[
        \{\partial_{b_i}U^b\}_{i=1,2,3}\subset\operatorname{Ker}(\mathcal{L}_{U^b})
    \]
    hence we have
    \[
        \operatorname{Ker}(\mathcal{L}_{U^b})\,\cap\,\{w|w\ \text{is self-similar}\}=\operatorname{span}\{\partial_{b_i}U^b\}_{i=1,2,3}
    \]
\end{proof}

As an application of the classification lemma of self-similar kernel of $\mathcal{L}_{U^b}$, we can prove the Thm \ref{thm:RSS} by compactness argument.

\begin{proof}[Proof for Theorem \ref{thm:RSS} when $\alpha$ small]\

\noindent\textbf{Step 1: Compactness method}

Let us proof by contradiction, suppose that there $\exists M_0>0$, $\alpha_n\to 0$ and a sequence $\{(u_n,p_n)\}_{n\ge 1}$ where each $(u_n,p_n)$ is an $\alpha_n$-RSS solution of SNS but \textbf{NOT} Landau solution, i.e.
\begin{equation}\label{eq:u_n_alpha_RSS}
    \left\{\begin{array}{c}
    -\Delta u_n+(u_n\cdot\nabla)u_n+\nabla p_n=\beta_ne_3\delta_0,\quad\ \operatorname{div}u_n=0\\
    u_n(x)=\lambda R(-2\alpha_n\log\lambda)u(\lambda R(2\alpha_n\log\lambda)x)
    \end{array}\right.
\end{equation}
and satisfies the uniform bound $\sup\limits_{x\in\mathbb{R}^3\backslash\{0\}}|x||u_n(x)|\le M_0$.

By standard scaling argument and regularity theory of SNS (ref \cite{ST2000}), we know that $|\nabla^ku_n|\le C_k(M_0)/|x|^{k+1}$ and $|\nabla^{k}p_n|\le C_k(M_0)/|x|^{k+2}$ for $\forall k\ge 0$ (zero order control follows from the RSS condition of $p_n$), denote that
\begin{equation}\label{eq:U_n_alpha_RSS}
    u_n(r,\phi,\theta)=\frac{1}{r} R(2\alpha_n\log r)U_n(\phi,\theta-2\alpha_n\log r)
\end{equation}
\begin{equation}\label{eq:P_n_alpha_RSS}
    p_n(r,\phi,\theta)=\frac{1}{r^2}P_n(\phi,\theta-2\alpha_n\log r)
\end{equation}
where $U_n(\phi,\theta)=u_n(1,\phi,\theta)$ and $P_n(\phi,\theta)=p_n(1,\phi,\theta)$, we know that $\{\nabla^k U_n\}_{n\ge 1}$ is uniformly bounded on $\mathbb{S}^2$ for $\forall k\ge 0$, hence $\exists$ subsequence of $U_n$ (still denote by $U_n$) and $U_{\infty}\in C^{\infty}(\mathbb{S}^2)$ such that $U_n$ converges uniformly to $U_{\infty}$ in the sense of $C^k(\mathbb{S}^2)$ for $\forall k\ge 0$. By similar argument, we have $ P_n\rightrightarrows P_{\infty}$ in $C^k(\mathbb{S}^2)$ for $\forall k\ge 0$.

Extend $U_{\infty}$, $P_{\infty}$ into $\mathbb{R}^3\backslash\{0\}$ in self-similar way
\[
    u_{\infty}(x)=\frac{1}{r}U_{\infty}(\sigma),\quad\ p_{\infty}(x)=\frac{1}{r^2}P_{\infty}(\sigma),\quad\ x=r\sigma
\]
by \eqref{eq:U_n_alpha_RSS} and \eqref{eq:P_n_alpha_RSS}, $u_n\rightrightarrows u_{\infty}$, $p_n\rightrightarrows p_{\infty}$ in $C^k(K)$ for $\forall k\ge 0$ and $K\Subset\mathbb{R}^3\backslash\{0\}$, taking limit on \eqref{eq:u_n_alpha_RSS}, we have
\[
    -\Delta u_{\infty}+(u_{\infty}\cdot\nabla)u_{\infty}+\nabla p_{\infty}=\beta_{\infty} e_3\delta_0
\]
where $\beta_n\to\beta_{\infty}$. Since $u_{\infty}$ is self-similar solution, $u_{\infty}$ must be Landau solution $U^{\beta_{\infty} e_3}$ with parameter $\beta_{\infty} e_3$. (To avoid ambiguity in notation, we denote the Landau solution with momentum strength $b$ to be $U^b(x)=\frac{1}{r}\widetilde{U}^b(\sigma)$).

\noindent\textbf{Step 2: Tangential limit}

Since $u_n$ is not Landau solution, $U_n\not\equiv \widetilde{U}^{\beta_n e_3}$, then we denote that
\[
    \varepsilon_n=\sup\limits_{\mathbb{S}^2}|U_n-\widetilde{U}^{\beta_ne_3}|>0,\quad\ \varepsilon_n\to 0
\]
(by $U_n\rightrightarrows \widetilde{U}^{\beta e_3}$, $\widetilde{U}^{\beta_n e_3}\rightrightarrows \widetilde{U}^{\beta e_3}$) and let 
\[
    W_n=\frac{1}{\varepsilon_n}(U_n-\widetilde{U}^{\beta_ne_3}),\quad\ \Pi_n=\frac{1}{\varepsilon_n}(P_n-\widetilde{P}^{\beta_ne_3}),\quad\ \|W_n\|_{L^{\infty}}= 1
\]

To taking tangential limit of $W_n$, we need to regain compactness, for this purpose, we extend again $W_n$ and $\Pi_n$ into $\mathbb{R}^3\backslash\{0\}$ by
\[
    w_n=\frac{1}{r}R(2\alpha_n\log r)W_n(\phi,\theta-2\alpha_n\log r)
\]
\[
    \pi_n=\frac{1}{r^2}\Pi_n(\phi,\theta-2\alpha_n\log r)
\]
then $w_n$ satisfies
\begin{equation}\label{eq:w_n_equ}
    -\Delta w_n+\nabla\pi_n=-\operatorname{div}(w_n\otimes U^{\beta_n e_3}+U^{\beta_ne_3}\otimes w_n+\varepsilon_n w_n\otimes w_n)
\end{equation}
using interior estimates of Stokes system on annulus $A_{1/4,4}$ and the fact that $U^b$ is bounded by some constant $C$ (only rely on $b$) in $A_{1/4,4}$, we have
\[
    \|\nabla w_n\|_{L^p(A_{1/2,2})}+\inf\limits_{s\in\mathbb{R}}\|\pi_n-s\|_{L^p(A_{1/2,2})}\le C_{p,b}(\varepsilon_n\|w_n\|_{L^\infty(A_{1/4,4})}^2+\|w_n\|_{L^{\infty}(A_{1/4,4})})\le CC_{p,b}
\]
from standard bootstrap argument $\|\nabla^k w_n\|_{L^\infty}\le C_k$, which implies $\|\nabla^kW_n\|_{L^\infty}\le C_k$, and consequently
\[
    \exists\,W_{\infty}\in C^{\infty}(\mathbb{S}^2),\quad\ W_n\rightrightarrows W_{\infty}\quad\ \text{in}\ C^k\ \text{sense},\ \forall k\ge 0
\]
similarly
\[
    \Pi_n\rightrightarrows\Pi_{\infty}\quad\ \text{in}\ C^k\ \text{sense}
\]

\noindent\textbf{Step 3: Rule out nontrivial tangential limit}

Extend $W_{\infty}$, $\Pi_{\infty}$ back to $\mathbb{R}^3\backslash\{0\}$ in self-similar way and get $w_{\infty}$ and $\pi_{\infty}$, since $(u_n,p_n)$ and $(U^{\beta_n e_3},P^{\beta_n e_3})$ share the same momentum flux input $\beta_n e_3\delta_0$, the equation \eqref{eq:w_n_equ} is actually valid in the whole $\mathbb{R}^3$ (especially at the origin), taking limit $n\to\infty$
\[
    -\Delta w_{\infty}+U^{\beta_{\infty} e_3}\cdot\nabla w_{\infty}+w_{\infty}\cdot\nabla U^{\beta_{\infty} e_3}+\nabla\pi_{\infty}=0 \quad\ \text{in}\ \mathbb{R}^3
\]
where $w_{\infty}$ is self-similar. Then using Prop 4.2, we can conclude that $w_{\infty}=0$, contradict to the fact that $\|W_{\infty}\|_{L^\infty}=1$
    
\end{proof}

\section{DSS cases and RDSS cases}

As a corollary of Prop \ref{prop:classification}, we can also proof the DSS and RDSS case in similar way.

\begin{proof}[Proof of Theorem \ref{thm:DSS}]
    Proof by contradiction as before, suppose that there $\exists M_0>0$, $\lambda_n\to 1^+$ and a sequence $\{(u_n,p_n)\}$ where each $(u_n,p_n)$ is $\lambda_n$-DSS solution of SNS but \textbf{NOT} Landau solution, with $|u_n(x)|\le\frac{M_0}{|x|}$.

    By similar compactness argument as in the last section, we have
    \[
        u_n(x)\rightrightarrows u_{\infty}(x)\quad\ \text{in}\ C^{\infty}_{\text{loc}}(\mathbb{R}^3\backslash\{0\})
    \]

    For $\forall\lambda\in\mathbb{R}_+$, since $\lambda_n\to 1^+$, we can find integers $k_n\in\mathbb{Z}$ such that $\lambda_n^{k_n}\to\lambda$, since for fix $x$, we have
    \[
        u_n(x)=\lambda_nu_n(\lambda_nx)=\lambda_n^{k_n}u_n(\lambda_n^{k_n}x)
    \]
    take $n\to\infty$ we know that $u_{\infty}(x)=\lambda u_{\infty}(\lambda x)$ for $\forall\lambda\in\mathbb{R}_+$, hence $u_{\infty}$ is a self-similar solution to SNS, must be a Landau solution $U^{b_\infty}$.

    Since $u_n$ is not Landau solution, we denote that
    \[
        \varepsilon_n=\sup\limits_{1\le|x|\le\lambda_n}|x||u_n(x)-U^{b_n}(x)|\to 0,\quad\ w_n=\frac{1}{\varepsilon_n}(u_n-U^{b_n})
    \]
    taking tangential limit as before then we can obtain $w_n\rightrightarrows w_{\infty}$ in $C^k(K)$ for $\forall k\ge 0$ and $K\Subset\mathbb{R}^3\backslash\{0\}$, where $w_{\infty}\in \operatorname{Ker}\mathcal{L}_{U^{b_{\infty}}}\cap\{w| w\ \text{is self-similar}\}$ with zero external force at origin, hence $w\equiv 0$, contradict to $\sup\limits_{x\in\mathbb{R}^3\backslash\{0\}}|x||w_{\infty}(x)|=1$!
\end{proof}
    
Similar argument in RDSS case
\begin{proof}[Proof of Theorem \ref{thm:RDSS}]
    Suppose that there $\exists M_0>0$, $\lambda_n\to 1^+$ and a sequence $\{(u_n,p_n)\}$ where each $(u_n,p_n)$ is $\lambda_n,\alpha$-RDSS solution of SNS but \textbf{NOT} Landau solution, with $|u_n(x)|\le\frac{M_0}{|x|}$.    

    If $\alpha\in 2\pi\mathbb{Q}$, by definition we know that all $\lambda_n,\alpha$-RDSS is actually DSS solution, so we only consider $\alpha\notin2\pi\mathbb{Q}$, by similar argument as DSS case, we obtain a limit $u\rightrightarrows u_{\infty}$ satisfies
    \[
        u_{\infty}(x)=\R(-\alpha)u_{\infty}(\R(\alpha)x)\ \Longrightarrow\ u_{\infty}(x)=\R(-k\alpha)u_{\infty}(\R(k\alpha)x)
    \]
    since $\{\,\{k\alpha/2\pi\}\,\}_{k\ge1}$ is dense in $[0,1)$ ($\{x\}$ is decimal part of $x$), hence $u_{\infty}(x)=\R (-\theta)u_{\infty}(\R(\theta)x)$ is axisymmetric.

    For $\forall\lambda\in\mathbb{R}_+$, $\exists k_n\in\mathbb{Z}$ such that $\lambda_n^{k_n}\to\lambda$, since 
    \[
        u_n(x)=\lambda_n\R(-\alpha)u_n(\lambda_n\R(\alpha)x)=\lambda_n^{k_n}\R(-k_n\alpha)u_n(\lambda_n^{k_n}\R(k_n\alpha)x)
    \]
    choose subsequence of $k_n$ (still denote by $k_n$) such that $k_n\alpha\to\theta\ \text{mod}\ 2\pi$, then for fix $x$, take $n\to\infty$, we have
    \[
        u_{\infty}(x)=\lambda\R(-\theta)u_{\infty}(\lambda\R(\theta)x)=\lambda u_{\infty}(\lambda x)
    \]
    where the last equality comes from axisymmetric property of $u_\infty$, therefore $u_\infty$ is self-similar, and remaining proof is just  a repetition of tangential limit argument in RSS and DSS cases. 
\end{proof}

\section{Energy estimate argument when $\alpha$ large}

By Thm \eqref{thm:DSS}, we know that $\exists\,\alpha_+(M)>0$ such that the only $\alpha$-RSS solution is Landau solution when $\alpha>\alpha_+(M)$. In this section we provide an alternative proof and establish the quantitative result $\alpha_+(M)\sim 1+O(M)$  

\subsection{Axisymmetric projection and commutation relations}

Firstly, we decompose $f,P\in C^{\infty}\left(\mathbb{S}^2\right)$ and $V\in \mathfrak{X}\left({\mathbb{S}^2}\right)$ into axisymmetric part and non-axisymmetric part, i.e.
\[
    f=\pa f+\pn f=\frac{1}{2\pi}\int_{0}^{2\pi}f(\theta,\phi)d\theta+\left(f-\frac{1}{2\pi}\int_{0}^{2\pi}f(\theta,\phi)d\theta\right)
\]
it corresponds to the zero mode and non-zero mode of Fourier expansion w.r.t. $\theta$.

For vector fields, we can also define this kind of decomposition, if $V=V_{\theta}\partial_{\theta}+V_{\varphi}\partial_\varphi$, then
\[
    V=\pa V+\pn V=((\pa V_{\theta})\partial_\theta+(\pa V_{\phi})\partial_\phi)+((\pn V_{\theta})\partial_\theta+(\pn V_{\phi})\partial_\phi)
\]

To apply the $\theta$-mode decomposition on the tangential and normal equation, we need to verify the commutation property between $\pa$, $\pn$ and differtial operator on $\mathbb{S}^2$ as follows:

\begin{lemma}
    {\ 
    
    (i) $\Lk\pa =0=\pa\Lk$, $K\pa=0=\pa K$, and $\Lk\pn=\pn\Lk(=\Lk)$, $K\pn=\pn K(=K)$ 
    
    (ii) $\pa$, $\pn$ commute with $\Lapsh$ and $\Laps$

    (iii) $\pa$, $\pn$ commute with $\nablas$ and $\divs$
    }
\end{lemma}
\begin{proof}
    (i) Suppose that $V=V^{\theta}\partial_\theta+V^{\phi}\partial_{\phi}$, then 
    \begin{equation}\label{eq:L_KV_formula}
        \Lk V=[K,V]=\partial_{\theta}V^{\theta}\partial_{\theta}+\partial_\theta V^{\phi}\partial_\phi
    \end{equation}
    since $\pa f$ is axisymmetric, we know that $K\pa f=0$, and by $\int_{0}^{2\pi}\partial_\theta f d\theta=0$, we have $\pa Kf=0$, similarly by \eqref{eq:L_KV_formula} we can prove that $\Lk\pa=0=\pa\Lk$. Consequently by $\mathbb{I}=\pa+\pn$, we get the last two formulas of (i).

    (ii) First we verify that $d$ commute with $\pa$
\begin{equation}\nonumber
\begin{aligned}
    d(\pa V^\flat)&=(\partial_\phi\pa \left(\sin^2\phi V^\theta\right))d\phi\wedge d\theta+(\partial_\theta\pa V^{\phi})d\theta\wedge d\phi\\
    &=\left(\frac{1}{2\pi}\int_{0}^{2\pi}\partial_\phi \left(\sin^2\phi V^\theta(\theta,\phi)\right)d\theta\right)d\phi\wedge d\theta =\left(\frac{1}{2\pi}\int_{0}^{2\pi}\partial_\phi \left(\sin^2\phi V^\theta\right)-\partial_\theta V^\phi)d\theta\right) d\phi\wedge d\theta\\
    &=\pa(dV^\flat)
\end{aligned}
\end{equation}

Since
\[
    *d\theta=-\frac{1}{\sin\phi}d\phi,\qquad *d\phi=\sin\phi d\theta,\qquad *d\theta\wedge d\phi=-\frac{1}{\sin\phi}
\]
(coefficients independent on $\theta$), we have $*\pa =\pa *$, then by $d^*=-*d\,*$ and $\Lapsh,\Laps=d^*d+dd^*$, we know that both $\Lapsh$ and $\Laps$ commute with $\pa$.

(iii) It follows from
\[
    \nablas(\pa f)=\frac{1}{\sin^2\phi}(\partial_\theta\pa f)\partial_\theta+(\partial_\phi\pa f)\partial_\phi=\pa\left(\nablas f\right)
\]
\[
    \divs (\pa V)=\partial_\theta(\pa V^{\theta})+\frac{1}{\sin\phi}\partial_\phi(\sin\phi\pa V^{\phi})=\pa\left(\divs V\right)
\]

\end{proof}

The following lemma is used for energy estimate in the next part after applying projection

\begin{lemma}\label{lem:integration_formula}\

     (i) \[\int_{\mathbb{S}^2}\pn V\cdot\pn Wd\sigma=\int_{\mathbb{S}^2}V\cdot\pn W d\sigma,\quad\ \int_{\mathbb{S}^2}\pn f\pn gd\sigma=\int_{\mathbb{S}^2}f\pn gd\sigma\]

     (ii) $\nablas_{\pa V}\pa V$ and $\nablas_{\pa V}\pa f$ is axisymmetric.
\end{lemma}

\begin{proof}
    (i) Follows from the facts that $\pa V\cdot\pn W$ and $\pa f\pn g$ has zero $\theta$-mean.

    (ii) Follows from $\nablas_{\partial_{\phi}}\partial_{\phi}=0$, $\nablas_{\partial_{\theta}}\partial_{\phi}=\nablas_{\partial_{\phi}}\partial_{\theta}=\cot\phi\partial_{\theta}$ and $\nablas_{\partial_{\theta}}\partial_{\theta}=-\sin\phi\cos\phi\partial_\phi$ are all axisymmetric, and $\pa V$, $\pa f$ is also axisymetric.
\end{proof}

\subsection{Energy estimate of profile equations}

Applying the non-axisymmetric projection on \eqref{eq:RSS_profile}, we have
\begin{equation}
    \left\{\begin{aligned}
        &\quad\quad\, -\Lapsh V_{\neq}-4\alpha^2\mathcal{L}_K^2V_{\neq}-2\alpha[(1+f)\mathcal{L}_KV]_{\neq}+\left(\nablas_VV\right)_{\neq}+\nablas(P_{\neq}-2f_{\neq})=0\\
        &-\Laps f_{\neq}-4\alpha^2K^2f_{\neq}-2\alpha[(f-1)Kf]_{\neq}+\left(\nablas_Vf-|V|^2-|f|^2\right)_{\neq}-2\alpha KP_{\neq}-2P_{\neq}=0\\
        &\qquad\qquad\qquad\qquad\qquad\qquad \divs V_{\neq}+f_{\neq}-2\alpha Kf_{\neq}=0
    \end{aligned}\right.
\end{equation}
where $V_{\neq},\ V_0,\ ...$ is abbreviation of $\pn V,\ \pa V,\ ...\ .$

The energy estimate of the tangential equation is
\begin{equation}\label{eq:tangential_estimate}
    \begin{aligned}
        &\;\|\nablas V_{\neq}\|_{L^2}^2+\|V_{\neq}\|_{L^2}^2+4\alpha^2\|\mathcal{L}_KV_{\neq}\|_{L^2}^2\\
        =\;&\int_{\mathbb{S}^2}2\alpha(1+f)V_{\neq}\mathcal{L}_K V_{\neq}
        -\left(\nablas_VV\right)_{\neq}V_{\neq}+\left(P_{\neq}-2f_{\neq}\right)\divs\left(V_{\neq}\right)d\sigma\\
        =\;&-\alpha\int_{\mathbb{S}^2}K(1+f)|V_{\neq}|^2d\sigma-\int_{\mathbb{S}^2}\left(\nablas_VV\right) V_{\neq}d\sigma+\int_{\mathbb{S}^2}(P_{\neq}-2f_{\neq})(2\alpha Kf_{\neq}-f_{\neq})d\sigma\\
        =\;&-\alpha\int_{\mathbb{S}^2}Kf_{\neq}|V_{\neq}|^2d\sigma-\int_{\mathbb{S}^2}\left(\nablas_{V}V\right)V_{\neq}d\sigma+2\|f_{\neq}\|_{L^2}^2+\int_{\mathbb{S}^2}P_{\neq}(2\alpha Kf_{\neq}-f_{\neq})d\sigma
    \end{aligned}
\end{equation}
using the fact that $K=\partial_{\theta}$, Lem \ref{lem:integration_formula} and \eqref{eq:Weitzenbock} in Appendix A. To eliminate the pressure term, we test the normal equation by $T_{\alpha}f=(2-2\alpha K)^{-1}(1-2\alpha K)f$.

\begin{lemma}
    {\rm Suppose $\alpha>0$, $K=\partial_\theta$, then $(2-2\alpha K):C^{\infty}(\mathbb{S}^2)\to C^{\infty}(\mathbb{S}^2)$ is an invertible operator and the operator $T_{\alpha}=(2-2\alpha K)^{-1}(1-2\alpha K)$ satisfies
    \[
        \frac{1}{2}\|f\|_{L^2}^2\le\langle T_\alpha f,f\rangle\le\|f\|_{L^2}^2,\quad\|T_{\alpha}f\|_{L^2}\le\|f\|_{L^2}
    \]
    }
\end{lemma}

\begin{proof}
    Take Fourier expansion in $\theta$-variable, for $\forall f\in C^{\infty}(\mathbb{S}^2)$, we have
    \[
        f(\theta,\phi)=\sum\limits_{m\in\mathbb{Z}}f_m(\phi)e^{im\theta}
    \]
    let
    \[
        g=\sum\limits_{m\in\mathbb{Z}}g_m(\phi)e^{im\theta},\quad g_{m}(\phi)=(2-2i\alpha m)^{-1}f_m(\phi)
    \]
    obviously $(2-2\alpha K)g=f$, $g$ is uniquely determined by $f$, so $(2-2\alpha K)$ is invertible.

    The Fourier multiplier correspond to $T_{\alpha}=(2-2\alpha K)^{-1}(1-2\alpha K)$ is
    \[
        z=\frac{1-2i\alpha m}{2-2i\alpha m},\qquad \frac{1}{2}\le\Re z=\frac{2+4\alpha^2m^2}{4+4\alpha^2m^2}\le 1
    \]
    which implies the inequality in the lemma.
\end{proof}

Since the multiplier $T_\alpha$ commute with $K$, $\nablas$, then the energy estimate of the normal equation is
\begin{equation}\label{eq:normal_estimate}
    \begin{aligned}
        &\;\frac{1}{2}\|\nablas f_{\neq}\|_{L^2}^2+2\alpha^2\|Kf_{\neq}\|_{L^2}^2\\ \le&\;\langle \nablas f_{\neq},T_{\alpha}\nablas f_{\neq}\rangle+4\alpha^2\langle Kf_{\neq},T_{\alpha}Kf_{\neq}\rangle\\
        =\;&\int_{\mathbb{S}^2}2\alpha(f-1)Kf_{\neq}\,T_{\alpha}f_{\neq}-\left(\nablas_Vf-|V|^2-|f|^2\right)T_{\alpha}f_{\neq}+P_{\neq}(1-2\alpha K)f_{\neq}d\sigma
    \end{aligned}
\end{equation}
then combining the two part, we have
\begin{equation}
    \begin{aligned}
        &\;\|\nablas V_{\neq}\|_{L^2}^2+\|V_{\neq}\|_{L^2}^2+4\alpha^2\|\mathcal{L}_KV_{\neq}\|_{L^2}^2+\frac{1}{2}\|\nablas f_{\neq}\|_{L^2}^2+2\alpha^2\|Kf_{\neq}\|_{L^2}^2\\
        \le &\;-\alpha\int_{\mathbb{S}^2}Kf_{\neq}|V_{\neq}|^2d\sigma-\int_{\mathbb{S}^2}\left(\nablas_{V}V\right)_{\neq}V_{\neq}d\sigma+2\|f_{\neq}\|_{L^2}^2+\int_{\mathbb{S}^2}2\alpha(f-1)Kf_{\neq}\,T_{\alpha}f_{\neq}d\sigma-\ \\
        &\qquad\qquad\qquad\qquad\qquad\qquad\qquad\qquad\qquad\qquad-\int_{\mathbb{S}^2}\left(\nablas_Vf-|V|^2-|f|^2\right)_{\neq}T_{\alpha}f_{\neq}d\sigma\\
        =&\; I_1+I_2+I_3+I_4+I_5
    \end{aligned}
\end{equation}
then we estimate $I_1,\ I_2,\ I_3,\ I_4,\ I_5$ term by term:

First of all, the key tool we use is a \textbf{Poincare type inequality} and the scale bound, i.e.
\[
    \|V_{\neq}\|_{L^2}\lesssim\|\mathcal{L}_KV_{\neq}\|_{L^2},\qquad\|f_{\neq}\|_{L^2}\lesssim\|Kf_{\neq}\|_{L^2}
\]
(since the zero-mode of $f_{\neq}$ and $V_{\neq}$ is 0) and from \eqref{eq:scale_bound} we have
\[
    |V|^2+|f|^2\le M^2,\qquad |V_0|,|f_0|\lesssim M,\ |V_{\neq}|,|f_{\neq}|\lesssim M
\]

\noindent\underline{\textbf{Estimate of $I_1$}:}

\begin{equation}
    |I_1|\le \alpha M\|Kf_{\neq}\|_{L^2}\|V_{\neq}\|_{L^2}\le\frac{1}{2}\alpha M\|Kf_{\neq}\|_{L^2}^2+\frac{1}{2}\alpha M\|\Lk V_{\neq}\|_{L^2}^2
\end{equation}

\noindent\underline{\textbf{Estimate of $I_2$}:}

\begin{equation}
    \begin{aligned}
        |I_2|\le &\;\left|\int_{\mathbb{S}^2}\left\langle\nablas_{V_0}V_{\neq},V_{\neq}\right\rangle d\sigma\right|+\left|\int_{\mathbb{S}^2}\left\langle\nablas_{V_{\neq}}V_{0},V_{\neq}\right\rangle d\sigma\right|+\left|\int_{\mathbb{S}^2}\left\langle\nablas_{V_{\neq}}V_{\neq},V_{\neq}\right\rangle d\sigma\right|\\
        \le &\; I_{2,1}+I_{2,2}+I_{2,3}
    \end{aligned}
\end{equation}
using the fact that $\nablas_{V_0}V_0$ is axisymmetric.

Since $V_0\langle V_{\neq},V_{\neq}\rangle=2\langle\nablas_{V_0}V_{\neq},V_{\neq}\rangle$
\[
    I_{2,1}=\frac{1}{2}\left|\int_{\mathbb{S}^2}V_0\langle V_{\neq},V_{\neq}\rangle d\sigma\right|=\frac{1}{2}\left|\int_{\mathbb{S}^2}|V_{\neq}|^2\divs V_0d\sigma\right|\le\frac{M}{2}\|V_{\neq}\|_{L^2}^2\lesssim M\|\mathcal{L}_KV_{\neq}\|_{L^2}^2
\]
using $\divs V_0=-f_0$.

By similar argument we can estimate the other terms:
\begin{equation}\nonumber
\begin{aligned}
    I_{2,2}\le\;&\left|\int_{\mathbb{S}^2}\left\langle V_0,\nablas_{V_{\neq}}V_{\neq}\right\rangle d\sigma\right|+\left|\int_{\mathbb{S}^2}\langle V_0,V_{\neq}\rangle\divs V_{\neq}\right|\\
    \le\;&M\|V_{\neq}\|_{L^2}\|\nablas V_{\neq}\|_{L^2}+M\|V_{\neq}\|_{L^2}\left(\|f_{\neq}\|_{L^2}+2\alpha\|Kf_{\neq}\|_{L^2}\right)\\
    \le\;&\frac{1}{2}\|\nablas V_{\neq}\|_{L^2}^2+C\left(2M^2+\frac{M}{2}+M\alpha\right)\|\mathcal{L}_KV_{\neq}\|_{L^2}^2+\left(\frac{CM}{2}+M\alpha\right)\|Kf_{\neq}\|_{L^2}^2
\end{aligned}
\end{equation}

\begin{equation}\nonumber
    \begin{aligned}
        I_{2,3}=\;&\frac{1}{2}\left|\int_{\mathbb{S}^2}|V_{\neq}|^2\divs V_{\neq}\right|
        \le\frac{M}{2}\|V_{\neq}\|_{L^2}\left(\|f_{\neq}\|_{L^2}+2\alpha\|Kf_{\neq}\|_{L^2}\right)\\
        \le\;&C\left(\frac{M}{4}+\frac{M\alpha}{2}\right)\|\mathcal{L}_KV_{\neq}\|_{L^2}^2+\left(\frac{CM}{4}+\frac{M\alpha}{2}\right)\|Kf_{\neq}\|_{L^2}^2
    \end{aligned}
\end{equation}

Then the estimate of $I_1$ is
\begin{equation}
    |I_2|\le\frac{1}{2}\|\nablas V_{\neq}\|_{L^2}^2+C(M\alpha+M^2+M)\|\Lk V_{\neq}\|_{L^2}^2+C(M\alpha+M)\|Kf_{\neq}\|_{L^2}^2
\end{equation}

\noindent\underline{\textbf{Estimate of $I_3$}:}

\begin{equation}
    |I_3|=2\|f_{\neq}\|_{L^2}^2\lesssim\|Kf_{\neq}\|_{L^2}^2
\end{equation}

\noindent\underline{\textbf{Estimate of $I_4$}:}

\begin{equation}
    |I_4|\le 2\alpha(M+1)\|Kf_{\neq}\|_{L^2}\|T_\alpha f_{\neq}\|_{L^2}\lesssim \alpha(M+1)\|Kf_{\neq}\|_{L^2}^2
\end{equation}

\noindent\underline{\textbf{Estimate of $I_5$}:}

\begin{equation}
\begin{aligned}
    |I_5|=\;&\left|\int_{\mathbb{S}^2}\left(\nablas_Vf-|V|^2-|f|^2\right)T_{\alpha}f_{\neq}d\sigma\right|\\
    \le\;&\left|\int_{\mathbb{S}^2}\left(\nablas_Vf\right)T_{\alpha}f_{\neq}d\sigma\right|+\left|\int_{\mathbb{S}^2}|V|^2 T_{\alpha}f_{\neq}d\sigma\right|+\left|\int_{\mathbb{S}^2}|f|^2T_{\alpha}f_{\neq}d\sigma\right|\\
    \le\;& I_{5,1}+I_{5,2}+I_{5,3}
\end{aligned}
\end{equation}

Similar as how we handle the convection term $\nablas_VV$, we have
\begin{equation}\nonumber
    \begin{aligned}
        I_{5,1}&=\left|\int_{\mathbb{S}^2}\left(\nablas_{V_0}f_{\neq}+\nablas_{V_{\neq}}f_0+\nablas_{V_{\neq}}f_{\neq}\right)T_{\alpha}f_{\neq}d\sigma\right|\\
        &\le \frac{1}{4}\|\nablas f_{\neq}\|_{L^2}^2+C(M^2+M+M\alpha)\|Kf_{\neq}\|_{L^2}^2+CM^2\|\Lk V_{\neq}\|_{L^2}^2
    \end{aligned}
\end{equation}
since $T_{\alpha}f_{\neq}$ have zero $\theta$-average and $|V_0|^2$ is axisymmetric, hence
\[
    I_{5,2}\le\int_{\mathbb{S}^2}(|V|+|V_0|)|V_{\neq}||T_{\alpha}f_{\neq}|d\sigma\le CM\|\Lk V_{\neq}\|_{L^2}^2+CM\|Kf_{\neq}\|_{L^2}^2
\]
similarly
\[
    I_{5,3}\le CM\|Kf_{\neq}\|_{L^2}^2
\]

\textbf{To conclude}, we have the energy estimate

\begin{equation}
    \begin{aligned}
        &\frac{1}{2}\|\nablas V_{\neq}\|_{L^2}^2+\|V_{\neq}\|_{L^2}^2+4\alpha^2\|\mathcal{L}_KV_{\neq}\|_{L^2}^2+\frac{1}{4}\|\nablas f_{\neq}\|_{L^2}^2+2\alpha^2\|Kf_{\neq}\|_{L^2}^2\\
        \le\;& C(M^2+M+M\alpha+\alpha)\left(\|\Lk V_{\neq}\|_{L^2}^2+\|Kf_{\neq}\|_{L^2}^2\right)
    \end{aligned}
\end{equation}
taking $\alpha\ge\alpha_-(M)\sim 1+O(M)$ sufficiently large the energy estimate tells us $V_{\neq}=f_{\neq}=0$, hence $V$ and $f$ are both axisymmetric, the rotation derivative term ($\Lk V$, $Kf$) in \eqref{eq:RSS_profile} vanish and the profile equation degenerate to the self-similar profile \eqref{eq:SS-profile}, by the result of Tian-Xin \cite{TX1998} or \v{S}ver\'{a}k \cite{S2011}, we know that $u=\frac{1}{r}(v+fn)$ must be Landau solution, now we finish the large $\alpha$ argument.

\appendix

\section{Appendix A: Geometric facts on $\mathbb{S}^2$}

The standard metric of $\mathbb{S}^2$ is 
\[
    g^{(\sigma)}=d\phi^2+\sin^2\phi d\theta^2
\]
with $\nabla^{(\sigma)}$ and $D$ is Levi-Cevita connection of $\mathbb{S}^2$ and $\mathbb{R}^3$, suppose that $V=V^{\theta}\partial_{\theta}+V^{\phi}\partial_{\phi}\in \mathfrak{X}(\mathbb{S}^2)$, $f\in C^{\infty}(\mathbb{S}^2)$ 
and $u=\frac{1}{r}(V+fn)$

\subsection{Reduce the equation of self-similar solution on $\mathbb{S}^2$}

The computation of convection term using Gauss and Weingarten formula of $\mathbb{S}^2\subset\mathbb{R}^3$
\[
    D_{X}Y=\frac{1}{r}\nablas_{X}Y-\frac{1}{r}\langle X,Y\rangle n,\quad\ D_{X}n=X\ \text{(restrict on sphere)}
\]
and $D_nX=D_nn=0$ since $X$, $n$ is independent of $r$, then
\[
    (u\cdot\nabla)u=D_{u}u=\frac{1}{r^3}\left(\nablas_VV-|V|^2n+\nablas_{V}fn-f^2n\right)
\]

For the viscous term, suppose that $e_\theta$, $e_\phi$ is local orthogonal frame on $\mathbb{S}^2$ 
\[
    \Delta=\partial_r^2+\frac{2}{r}\partial_r+\frac{1}{r^2}\Laps_S
\]
where $\Laps_S=D_{e_\phi}D_{e_{\phi}}+D_{e_{\theta}}D_{e_\theta}$ is spherical Laplacian, and using Gauss formula, we have
\begin{equation}
    \begin{aligned}
        \Laps_S V & =D_{e_\phi}D_{e_{\phi}}V+D_{e_{\theta}}D_{e_\theta}V=\sum\limits_{\alpha\in\{\phi,\theta\}}\nablas_{e_\alpha}\nablas_{e_{\alpha}}V-\langle V,e_{\alpha}\rangle e_{\alpha}-2\langle\nablas_{e_{\alpha}}V,e_{\alpha}\rangle n\\
        & = \Laps_{C}V-V-2\left(\divs V\right)n
    \end{aligned}
\end{equation}
using musical isomorphism ($V^\flat=\sin^2\phi V^{\theta}d\theta+V^{\phi}d\phi$ and $(V^{\flat})^{\sharp}=V$), we have $\left(\Laps_C V\right)^\flat=\Laps_CV^\flat$, by Weitzenb\"{o}ck formula $\Laps_C=\Lapsh+1$ (where $\Lapsh=dd^*+d^*d$ is Hodge Laplacian), we have
\[
    \Laps_S V=\left(\Lapsh V^\flat\right)^{\sharp}-2\left(\divs V\right)n
\]
for convenience we denote $\left(\Lapsh V^\flat\right)^{\sharp}=:\Lapsh V$.

Similarly $\Laps_S(fn)=(\Laps f)n+2\nablas f-2fn$, hence
\[
    \Delta u=\frac{1}{r^3}\left(\Lapsh V+2\nablas f+(\Laps f)n\right)
\]
and the computation of $\nabla p$ and $\operatorname{div}u=0$ is standard.

When doing energy estimate, we will use
\begin{equation}\label{eq:Weitzenbock}
    -\int_{\mathbb{S}^2}V\cdot\Lapsh V=\|\nablas V\|_{L^2}^2+\|V\|_{L^2}^2
\end{equation}
which is from Weitzenb\"{o}ck formula.

\subsection{Vorticity equation}

For $V\in\mathfrak{X}(\mathbb{S}^2)$, its vorticity on $\mathbb{S}^2$ is given by $dV^{\flat}=\omega\Omega_0$, where $\Omega_0$ is volume form of $\mathbb{S}^2$, then the vorticity of viscosity term is
\[
    *d(\Lapsh V)=*d(dd^*+d^*d)V^{\flat}=\Laps\omega
\]
and the pressure term is killed by $d(\nablas(p-2f))^\flat=0$.

When it comes to convection term, for the use in Prop \ref{prop:classification}, we need the following fact

\begin{proposition}\label{prop:perturb_vorticity_equation}
    \begin{equation}\label{eq:perturb_vorticity_equation}
        d\left(\nabla_XY+\nabla_YX\right)^\flat=\mathcal{L}_X(dY^\flat)+\mathcal{L}_Y(dX^\flat)
    \end{equation}
\end{proposition}
\begin{proof}
Since $(\mathcal{L}_X\alpha)(Y)=X(\alpha(Y))-\alpha([X,Y])$
    \begin{equation}\nonumber
    \begin{aligned}
        \left(\mathcal{L}_XY^\flat\right)(Z)&=X(g(Y,Z))-g(Y,[X,Z])=g(\nabla_XY,Z)+g(Y,\nabla_XZ)-g(Y,[X,Z])\\
        &=(\nabla_XY)^\flat(Z)+g(Y,\nabla_ZX)
    \end{aligned}
    \end{equation}

Similarly
\[
    \left(\mathcal{L}_YX^\flat\right)(Z)=(\nabla_YX)^\flat(Z)+g(X,\nabla_ZY)
\]
then
\[
    \left(\mathcal{L}_XY^\flat+\mathcal{L}_YX^\flat\right)(Z)=\left(\nabla_XY+\nabla_YX\right)^{\flat}(Z)+Z(g(X,Y))=\left(\nabla_XY+\nabla_YX\right)^{\flat}(Z)+d\langle X,Y\rangle(Z)
\]
hence
\begin{equation}
    d\left(\nabla_XY+\nabla_YX\right)^\flat=d\left(\mathcal{L}_XY^\flat+\mathcal{L}_YX^\flat\right)=\mathcal{L}_X(dY^\flat)+\mathcal{L}_Y(dX^\flat)
\end{equation}
\end{proof}

By Prop \ref{prop:perturb_vorticity_equation}, we know $d(\nablas_VV)^\flat=\mathcal{L}_VdV^\flat=(\mathcal{L}_V\omega)\Omega_0+\omega\divs V\Omega_0=\divs(V\omega)\Omega_0$, hence the vorticity equation of tangential equation in \eqref{eq:SS-profile} is \eqref{eq:vorticity_SS}, and the vorticity equation of the tangential part of linearized equation \eqref{eq:perturb_equation} is \eqref{eq:vorticity_linear} since
\[
    d(\nablas_Vv+\nablas_vV)^\flat=\mathcal{L}_V(dv^\flat)+\mathcal{L}_v(dV^\flat)=\divs(V\omega)
\]
the last equality uses $dV^\flat=0$, where $V$ is a Landau solution.

\section{Appendix B: Maximal principle type lemma}

To prove vorticity equation \eqref{eq:vorticity_SS} forces $\omega\equiv 0$, \v{S}ver\'ak \cite{S2011} prove the following maximal principle type lemma (since the zero order term in $L$ has no sign):

\begin{lemma}\label{lem:maximal_lemma}
    Suppose that $V\in \mathfrak{X}\left(\mathbb{S}^2\right)$, $f\in C^{\infty}(\mathbb{S}^2)$ satisfies
    \[
        -\Laps f+\divs(fV)=0,\qquad\int_{\mathbb{S}^2}fd\sigma=0
    \]
    then $f\equiv 0$.
\end{lemma}
\begin{proof}
    Denote
    \[
        L=-\Laps+\divs(\cdot\, V),\quad\  L^*=-\Laps-\nablas_V
    \]
    
    By strong maximal principle on compact manifold we know that
    \[
        \operatorname{Ker} L^*=\{\text{const}\}
    \]
    since $\divs(\cdot\,V)$ is a compact perturbation of $-\Laps$, then from Fredholm index theory
    \[
        0=\operatorname{ind}\left(-\Laps\right)=\operatorname{ind}L=\operatorname{dim}\operatorname{Ker}L-\operatorname{dim}\operatorname{Ker}L^*
    \]
    therefore $\operatorname{Ker}L$ is one dimensional.

    Suppose that $\operatorname{Ker} L=\{f_0\}$, where $f_0\not\equiv 0$, we claim that $f_0$ does not change sign on $\mathbb{S}^2$, if $f_0$ change sign, then $\exists\, h(>0)\in C^{\infty}(\mathbb{S}^2)$, such that
    \[
        \int_{\mathbb{S}^2}f_0hd\sigma=0
    \]
    since $\{f_0\}=\operatorname{Ker} L$, $h\perp f_0$ $\Longrightarrow$ $h\in \operatorname{Im}L^*$, i.e. $\exists\, f_1\in C^{\infty}(\mathbb{S}^2)$ such that
    \[
        L^*f_1=h
    \]
    since $h$ positive, by strong maximal principle $f_1\equiv\text{const}$, which contradicts with $h>0$, hence $f_0$ dose not change sign on $\mathbb{S}^2$, $f=cf_0$ but
    \[
        \int_{\mathbb{S}^2}fd\sigma=0
    \]
    then $f\equiv 0$.
\end{proof}

\noindent\textbf{Acknowledgments:} L. Zhang is partially supported by NSFC under grant 12484540, 12494541 and National Key RD program of China Grant No.2021YFA1000800.

\bibliographystyle{abbrv}

\vspace{2em}
\noindent Institute of Mathematics, AMSS, UCAS, Beijing 100190, China
123\\
\texttt{{\em Email address:}shiyuxuan@amss.ac.cn}\\

\noindent Hua Loo-Keng Key Laboratory of Mathematics, Institute of Mathematics, AMSS, and School of
Mathematical Sciences, UCAS, Beijing 100190, China 
123\\
\texttt{{\em Email address:}lqzhang@math.ac.cn}

\end{document}